\documentclass[11pt]{amsart}
\usepackage{graphicx}
\usepackage{amsmath}
\usepackage{amssymb}
\usepackage{color}
\usepackage{mathdots}
\usepackage{bbm}
\usepackage{amsfonts}
\usepackage{graphicx}
\usepackage{mathrsfs}
\usepackage{mathtools}

\newtheorem{thm}{Theorem}[section]
\newtheorem{cor}[thm]{Corollary}

\theoremstyle{definition}

\theoremstyle{remark}
\newtheorem{rem}[thm]{Remark}

\numberwithin{equation}{section}

\newcommand{\cI}{\mathcal{I}}

\newcommand{\cV}{\mathcal{V}}

\newcommand{\RR}{\mathbb{R}}

\newcommand{\NN}{\mathbb{N}}

\newcommand{\supp}{{\rm supp \,}}

\begin{document}
\title[The moment problem on curves with bumps]{The moment problem on curves with bumps} %

\author[DP Kimsey]{David P Kimsey}
\address{School of Mathematics, Statistics and Physics\\
Newcastle University\\
Newcastle upon Tyne NE1 7RU UK}
\email{david.kimsey@ncl.ac.uk}
\author[M Putinar]{Mihai Putinar}
\address{Department of Mathematics \\
University of California Santa Barbara \\
Santa Barbara, CA 93106-3080 USA and School of Mathematics, Statistics and Physics \\
Newcastle University\\
Newcastle upon Tyne NE1 7RU UK
}
\email{mihai.putinar@ncl.ac.uk}
\thanks{The second author would like to thank the Isaac Newton Institute for Mathematical Sciences for support and hospitality during the program Complex Analysis, Fall 2019, when work on this paper was undertaken. This work was supported by
EPSRC grant number EP/R014604/1. Both authors are thankful to Professor Daniel Plaumann and Professor Claus Scheiderer for pointing out item (c) in Remark \ref{rem:FINAL} and also to Professor Markus Schweighofer for carefully reading the paper and making constructive suggestions.}

\begin{abstract} The power moments of a positive measure on the real line or the circle are characterized by the non-negativity of an infinite matrix, Hankel, respectively Toeplitz, attached to the data.
Except some fortunate configurations, in higher dimensions there are no non-negativity criteria for the power moments of a measure to be supported by a prescribed closed set. We combine two well studied fortunate situations, specifically a class of curves in two dimensions classified by Scheiderer and Plaumann, and compact, basic semi-algebraic sets, with the aim at enlarging the realm of geometric shapes on which the power moment problem is accessible and solvable by non-negativity certificates.
\end{abstract}

\maketitle

\section{Introduction}
Throughout the present note $\RR[x_1, \ldots, x_d]$ denotes the ring of polynomials with real coefficients in $d$ indeterminates. We adopt the standard notation  
$$x^{\gamma} = \prod_{j=1}^d x_j^{\gamma_j} \quad {\rm and} \quad |x| := \sqrt{x_1^2 + \ldots + x_d^2 },$$  
where $x = (x_1, \ldots, x_d) \in \RR^d$ and $\gamma = (\gamma_1, \ldots, \gamma_d) \in \NN_0^d$. The convex cone of polynomials $p \in \RR[x_1, \ldots, x_d]$ which can written as a {\it sum of squares} is $\Sigma^2$. The elements of $\Sigma^2$ represent universally non-negative polynomials.
The real zero set of the ideal $\cI := (p_1, \ldots, p_k)$ generated by $p_1, \ldots, p_k$ in $\RR[x_1, \ldots, x_d]$ is 
$$\cV(\cI):= \{ x \in \RR^d: p_1(x) = \ldots = p_k(x) = 0 \}.$$

Recalling some basic notions of real algebraic geometry is also in order. Specifically, for a finite subset $R = \{ r_1, \ldots, r_k \} \subseteq \RR[x_1, \ldots, x_d]$, we let $Q_R$ stand for the {\it quadratic module} generated by $R$:
$$Q_R = \{ \sigma_0 + r_1 \, \sigma_1 + \ldots + r_k \, \sigma_k: \sigma_0, \ldots, \sigma_k \in \Sigma^2\}.$$ Also,
$$K_Q := \{ x \in \RR^d: r_j(x) \geq 0 \quad {\rm for} \quad j=1, \ldots, k \}$$
is the common non-negativity set of elements of $Q = Q_R$. In general a quadratic module is a subset of the polynomial algebra closed under addition and multiplication by sums of squares, see \cite{Prestel-Delzell}.

Given a multisequence $s = (s_{\gamma})_{\gamma \in \NN_0^d}$ and a closed set $K \subseteq \RR^d$, the {\it full $K$-moment problem on $\RR^d$} entails determining whether or not there exists a positive Borel measure $\mu$ on $\RR^d$ such that
\begin{equation}
s_{\gamma} = \int_{\RR^d} x^{\gamma} d\mu(x) \quad \quad {\rm for} \quad \gamma  \in \NN_0^d \label{eq:8Augm1}
\end{equation}
and
\begin{equation}
\label{eq:8Augm2}
\supp \mu \subseteq K.
\end{equation}
If conditions \eqref{eq:8Augm1} and \eqref{eq:8Augm2} are satisfied, then we say that $s$ has a $K$-representing measure.

A multisequence $s = (s_{\gamma})_{\gamma \in \NN_0^d}$ is called {\it positive definite} if
$$L_s(f) \geq 0 \quad \quad {\rm for} \quad f \in \Sigma^2.$$
It is clear that the Riesz-Haviland functional $L_s$ is non-negative on the quadratic module $Q$, whenever the moment problem with $s$ has a $K_Q$-representing measure, where
$$K_Q = \{ x \in \RR^d: r(x) \geq 0 \quad {\rm for}  \quad r \in Q \}.$$ 
Whether the converse is true is one of the central questions of multivariate moment problem theory, see \cite{Prestel-Delzell,Schmuedgen_book} for ample details. In this direction we recall a useful terminology.
A quadratic module $Q$ is said to satisfy the {\it strong moment property} (SMP) if every $Q$-positive functional $L: \RR[x_1, \ldots, x_d] \to \RR$ is a moment functional with the additional requirement that the measure is supported on $K_Q$, i.e., there exists a positive Borel measure $\mu$ supported by $K_Q$ such that
$$L(f) = \int f(x) d\mu(x) \quad\quad {\rm for} \quad f \in \RR[x_1, \ldots, x_d].$$

When dropping the requirement ${\rm supp} (\mu) \subseteq K_Q$ in the (SMP), we simply say that $Q$ possesses the moment property (MP). A quadratic module $Q$ is called {\it archimedean} if there exists a positive constant $C$ with the property $C-|x|^2 \in Q$. In this case $K_Q$ is compact and, by an observation of the second author, the module $Q$ has property (SMP), see again \cite{Prestel-Delzell,Schmuedgen_book} for details.

A classical theorem due to Hamburger (see \cite{Hamburger1,Hamburger3,Hamburger2} and \cite{Schmuedgen_book} for a contemporary treatment) asserts that on the real line every positive definite sequence has the strong moment property. In equivalent terms, the non-negativity of the infinite Hankel matrix
$(s_{k+n})_{k,n=0}^\infty$ is necessary and sufficient for $(s_n)_{n=0}^\infty$ to be the power moment sequence of a positive measure on $\RR$.

Two dimensions are special, notably for allowing to extend similar sufficient positivity conditions for the solvability of the moment problem along codimension-one unbounded varieties, that is real algebraic curves. Note that if $q \in \RR[x_1, x_2]$ is non-zero, than $\cV(q)$ is a curve, or a set of real points. 

%

One step further, we are seeking only reduced principal ideals, that is we enforce that a polynomial $f$ vanishes on $\cV(q)$ if and only if $f \in (q)$. This happens if the factorization of $q$ into irreducible factors is square free and each factor changes sign in $\RR^2$. See \cite{BCR}  for a proof and the natural framework for such a real Nullstellensatz. In this scenario we simply say that $(q)$ is a {\it real ideal}.
The main results of \cite{Scheiderer_curves} and \cite{Plaumann} may be combined to produce the following theorem.

\begin{thm}[\cite{Scheiderer_curves}, \cite{Plaumann}]
\label{thm:SP}
Let $(q)$ be a non-trivial, real principal ideal in $\RR[x_1, x_2]$.  Then
$$(q) + \Sigma^2 = \{ p \in \RR[x_1, x_2]: \text{$p(x) \geq 0$ for all $x \in \cV(q)$}\}$$
if and only if the following conditions hold{\rm :}
\begin{enumerate}
\item[{\rm (i)}] All real singularities of $\cV(q)$ are ordinary multiple points with independent tangents.
\item[{\rm (ii)}] All intersection points of $\cV(q)$ are real.
\item[{\rm (iii)}] All irreducible components of $\cV(q)'$ (i.e., the union of all irreducible components of $\cV(q)$ that do not admit any non-constant bounded polynomial functions) are non-singular and rational.
\item[{\rm (iv)}] The configuration of all irreducible components of $\cV(q)'$ contains no loops.
\end{enumerate}
\end{thm}

In particular, the above result implies that the quadratic module $(q) + \Sigma^2$ has the strong moment property \cite{Scheiderer_curves,Plaumann}. It is also worth mentioning that \cite{Netzer} has discovered finitely generated preorderings in $\RR[x_1, x_2]$ which satisfy the strong moment property. The above result is in sharp contrast to higher dimensional situations, where in general not every positive definite functional along a variety is represented by integration against a positive measure (see, \cite{Scheiderer_surfaces} for details).

\section{Main result}

We consider the union of a curve which satisfies conditions (i)-(iv) in Theorem \ref{thm:SP} with a side (to become clear in an instant) of a truly compact semi-algebraic set with the aim at providing positivity certificates 
for the moment problem to be solvable on that prescribed support.

\begin{thm}
\label{thm:10Septm1}

Let $(q)$ be a non-trivial, real principal ideal of $\RR[x_1, x_2]$ whose zero set satisfies conditions (i)-(iv) in Theorem \ref{thm:SP} and let $Q \subseteq \RR[x_1, x_2]$ be an archimedean quadratic module. Then
the quadratic module $ \Sigma^2 + qQ$ has the strong moment property.
\end{thm}

Before proving Theorem \ref{thm:10Septm1}, we pause to note that the positivity set of $ \Sigma^2 + qQ$ is $\cV(q) \cup [K_Q \cap \{ q > 0\}].$  For instance, taking $q(x_1,x_2) = x_1$ and $Q$ generated by $1-x_1^2-x_2^2$ one finds the positivity set
of the composed quadratic module to be the $x_2$-axis union with the half-disk $\{ (x_1,x_2), x_1 \geq 0, x_1^2 + x_2^2 \leq 1\},$
whence the title of this note.

\begin{figure}[h!]
\includegraphics[width=0.5\textwidth]{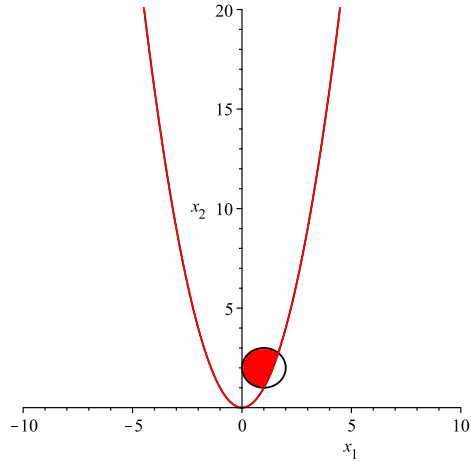}
  \caption{$Q$ is the quadratic module generated by $ \{ 1 - (x_1-1)^2 - (x_2-2)^2 \}$  and $q(x_1,x_2) = x_2 - x_1^2$}
\end{figure}

\begin{figure}[h!]
\includegraphics[width=0.5\textwidth]{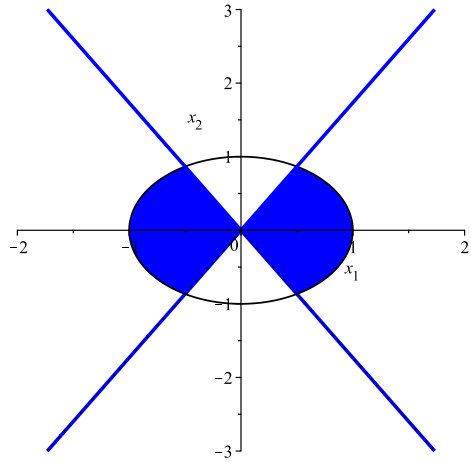}
  \caption{$Q$ is the quadratic module generated by $ \{ 1 - x_1^2 -x_2^2 \}$  and $q(x_1,x_2) = 3x_1^2-x_2^2$}
\end{figure}

\newpage

\begin{proof}[Proof of Theorem \ref{thm:10Septm1}.] We denote in short $x = (x_1,x_2)$.
Let $L \in \RR[x]'$ be a non-trivial linear functional which is non-negative on $ \Sigma^2 + qQ$. We want to prove that $L$ is represented by integration against a positive measure supported by $\cV(q) \cup [K_Q \cap \{ q \geq 0\}].$ Since $L$ is non-zero, the Cauchy-Schwarz inequality
$$ L(f)^2 \leq L(f^2) L(1), \ \ f \in \RR[x],$$
implies $L(1)>0$. Below we will use repeatedly the observation that there are elements $f \in\RR[x]$ with the property $L(f)>0$.

The functional $h \mapsto L(q h)$ is non-negative for $h \in Q$ and $Q$ is an archimedean quadratic module, so there exists a positive Borel measure $\nu$ supported by $K_Q$, such that:
$$ L(qf) = \int_{K_Q} f d\nu, \ \ f \in \RR[x],$$
see \cite{Prestel-Delzell,Schmuedgen_book}.

We claim that the measure $\nu$ does not carry mass on the set $\{ q \leq 0\}$, i.e.,
\begin{equation}
\label{eq:CLAIM}
 \nu ( \{ q \leq 0 \}) = 0.
\end{equation}

The positivity of the functional $L$ on squares yields
$$ L(\{ t q g +f\}^2) \geq 0 \quad t \in \RR \quad {\rm and} \quad f,g \in \RR[x].$$
On the other hand,
$$L(\{ t q g + f \}^2) = t^2 L(q^2 g^2) + 2 t L(q f g) + L(f^2),$$
hence
\begin{equation}
\label{eq:BR}
\left(\int_{K_Q} f(x) g(x) \, d\nu(x) \right)^2 \leq \left( \int_{K_Q} q(x) g(x)^2 \, d\nu(x) \right) \, L(f^2)
\end{equation}
for $f, g \in \RR[x]$. 

The non-negativity set $K_Q$ is compact, therefore continuous functions on $K_Q$ can be uniformly approximated by polynomials. Moreover, continuous functions on $K_Q$
are dense in $L^2(\nu)$. That is 
\begin{equation}
\label{eq:BR2}
\left( \int_{K_Q} f(x) \psi(x) \, d\nu(x) \right)^2 \leq \left( \int_{K_Q} q(x) \psi(x)^2 \, d\nu(x) \right) \, L(f^2),
\end{equation}
where $f \in \RR[x]$ and $\psi \in L^2(\nu)$. If we let $\chi = \mathbbm{1}_{\cV(\cI) \cap K_Q}$ denote the characteristic function of $\cV(\cI) \cap K_Q$ and $\psi = \chi$, then \eqref{eq:BR2} becomes
$$
\left( \int_{K_Q} f(x) \chi(x) \, d\nu(x) \right)^2 \leq \left( \int_{K_Q} q(x) \chi(x) \, d\nu(x) \right) \, L(f^2).
$$
From  $q(x) \chi(x) = 0$ we infer 
$$ \int_{\cV(q) \cap K_Q} f(x) d\nu(x) = 0, \ \ f \in \RR[x].$$

Choosing next $\psi$ to be the characteristic function of a compact subset of the open set $\{ q<0\}$, we find from \eqref{eq:BR2}:
$$ 0 \leq \left( \int_{K_Q} q(x) \chi(x) \, d\nu(x) \right) \, L(f^2). $$
In particular,
$$ 0 \leq \int_{K_Q} q(x) \chi(x) \, d\nu(x) \leq 0,$$
for every characteristic function of a compact subset of $\{ q<0\}$. This proves \eqref{eq:CLAIM}.

Next we choose $\psi$ in \eqref{eq:BR2} of the form $\psi = \frac{f}{q} \phi$, where $\phi = \phi^2$ is the characteristic function of a compact subset of $ \{ q > 0 \}$. We find
$$ \left( \int_{K_Q} f(x) \left(\frac{f(x)}{q(x)} \right) \phi(x) \, d\nu(x) \right)^2 \leq \left( \int_{K_Q} q(x) \left( \frac{f(x)^2}{q(x)^2}  \right) \phi(x) \, d\nu(x) \right) \, L(f^2),$$
or equivalently, since $q > 0$ on the support of $\phi$:
$$ \left( \int_{K_Q} \left( \frac{f(x)^2}{q(x)} \right) \phi(x) \, d\nu(x) \right)^2 \leq \left( \int_{K_Q} \left( \frac{f(x)^2}{q(x)} \right) \phi(x) \, d\nu(x) \right) \, L(f^2).$$
But $\left(\frac{f^2}{q} \right) \phi \geq 0,$ hence
$$ \int_{K_Q} \left(\frac{f(x)^2}{q(x)} \right) \phi(x) \, d\nu(x) \leq L(f^2).$$
A monotonic sequence of such characteristic functions $\phi$ converging point wise to the characteristic function of $\{ q>0\}$ implies $\frac{1}{q} \in L^1(\nu)$. Recall that $L(1)>0$.

Let $\Lambda: \RR[x] \to \RR$ denote the linear functional
\begin{equation}
\label{eq:Lambda}
\Lambda(f) = L(f) - \int_{K_Q} \left(\frac{f(x)}{q(x)} \right) \, d\nu(x), \quad \quad  f \in \RR[x].
\end{equation}
We claim that
\begin{equation}
\label{eq:L1}
\Lambda(q \, f) = 0 \quad\quad {\rm for} \quad f \in \RR[x]
\end{equation}
and
\begin{equation}
\label{eq:L2}
\Lambda(f^2) \geq 0 \quad\quad {\rm for} \quad f \in \RR[x].
\end{equation}
Assertion \eqref{eq:L1} follows immediately from the definition of $\Lambda$. We will now verify \eqref{eq:L2}. 
Given the $\nu$-integrability of $\frac{1}{q}$ we can choose $\psi = f/q$ in \eqref{eq:BR2}. This yields
$$\left( \int_{K_Q} \frac{f(x)^2}{q(x)} \, d\nu(x) \right)^2 \leq \left( \int_{K_Q} \frac{f(x)^2}{q(x)} \, d\nu(x) \right) \, L(f^2)$$
for $f \in \RR[x]$. Since $\frac{f(x)^2}{q(x)} \geq 0 $ on the support of the measure $\nu$, we find
$$
L(f^2) \geq \int_{K_Q} \left( \frac{f(x)^2}{q(x)} \right) \, d\nu(x) \quad\quad {\rm for}  \quad f \in \RR[x]
$$
which is exactly \eqref{eq:L2}. 

Finally, because the ideal $(q)$ is real and its zero set satisfies conditions (i)-(iv) in Theorem \ref{thm:SP} $\Lambda$ has a representing measure supported on $\cV(q)$.
This proves that the integration against the measure $\mu = \frac{\nu}{q} + \sigma$ represents the original functional $L$. Moreover, the support of $\mu$ in contained in the union of the supports of 
$\sigma$ and $\nu$, that is ${\rm supp} \mu \subset \cV(q) \cup [K_Q \cap \{ q > 0\}].$ 
\end{proof}

Given a bisequence $s = (s_{\gamma} )_{\gamma \in \NN_0^2 }$ and $p(x) = \sum_{0 \leq |\lambda| \leq n} p_{\lambda} x^{\lambda} \in \RR[x_1, x_2]$, we shall let $p(E)s$ denote the bisequence given by
$$(p(E)s)(\gamma) := \sum_{0 \leq |\lambda| \leq n } p_{\lambda} s_{\lambda + \gamma} \quad \quad {\rm for} \quad \gamma \in \NN_0^2.$$
\begin{cor}
\label{cor:C1}
Let $s = (s_{\gamma_1, \gamma_2})_{(\gamma_1, \gamma_2) \in \NN_0^2}$ be a positive definite bisequence and let $Q = Q(r_1, \ldots, r_k) \subseteq \RR[x_1, x_2]$ be an archimedean quadratic module. If there exists $q \in \RR[x_1, x_2]$ such that $(q)$ is a non-trivial, real principal ideal of $\RR[x_1, x_2]$ whose zero set satisfies conditions (i)-(iv) of Theorem \ref{thm:SP} and
\begin{equation}
\label{eq:P1}
\text{$q \, r_j(E) s$ is positive  definite for $j = 1 ,\ldots, k$,}
\end{equation}
then $s$ has a representing measure $\mu$ with
$$\supp \mu \subseteq \cV(q) \cup [ K_{Q} \cap \{ q > 0 \} ].$$
\end{cor}

\begin{proof}
Let $L_s: \RR[x_1, x_2] \to \RR$ denote the Riesz-Haviland functional with respect to $s$. Then, since $s$ is positive definite and we have a suitable $q \in \RR[x_1, x_2]$ such that \eqref{eq:P1} is in force, we have
$$L_s(f^2 + q\, g ) \geq 0 \quad\quad {\rm for} \quad f \in \RR[x_1, x_2] \quad {\rm and} \quad g \in Q.$$
Thus, the desired conclusion follows immediately from Theorem \ref{thm:10Septm1}.
\end{proof}

We add a few remarks on Theorem \ref{thm:10Septm1} and its proof.

\begin{rem}
\label{rem:FINAL}
\begin{enumerate}
\item[(a)] If the quadratic module in the statement of the theorem is finitely generated $Q = Q(r_1, \ldots, r_k),$ then the enhanced quadratic module which is shown to carry (SMP)
is $Q(q, qr_1, \ldots, q r_k)$. Notice that the latter may not be archimedean, although $Q$ is.

\item[(b)] Changing the generator of the principal ideal $(q)$ will alter the outcome of the statement, for instance $-q$ instead of $q$ in the enhanced quadratic module will flip the ``bumps" on the other side of the curve
$\cV(q)$.

\item[(c)] In the case that $Q$ is a preordering of $\RR[x_1, x_2]$, then Theorem \ref{thm:10Septm1} follows immediately from Schm\"udgen's fibre theorem \cite{Schmuedgen}. Indeed, we note that, since $K_Q$ is compact, $q$ is a bounded polynomial on $\cV(q) \cup [ K_Q \cap \{ q > 0 \} ]$ and the preordering $Q_\lambda \subseteq \RR[x_1, x_2]$ given by
$$Q_{\lambda} := Q + (q - \lambda) \quad \quad {\rm for} \quad 0 \leq \lambda \leq M,$$
where $(q-\lambda)$ denotes the ideal generated by the polynomial $q(x_1,x_2) - \lambda$ and
$$M = \max_{ (x_1, x_2) \in \cV(q) \cup [ K_Q \cap \{ q > 0 \} ]  } \, q(x_1,x_2),$$
satisfies the (SMP). The preordering $Q_{0}$ satisfies the (SMP) by Theorem \ref{thm:10Septm1} and $Q_{\lambda}$ with $\lambda > 0$ satisfies the (SMP) since $K_{Q_{\lambda}}$ is compact. 

\end{enumerate}
\end{rem}

The statement of Theorem \ref{thm:10Septm1} and Corollary \ref{cor:C1} can be generalized to any number of variables, keeping $(q)$ a real ideal with its zero set hypersurface possessing the (SMP), i.e., the quadratic module $\Sigma^2 + (q) \subseteq \RR[x_1, \ldots, x_d]$ has the strong moment property. This is the case for instance of a
compact zero set $\cV(q)$. Indeed, if $\cV(q)$ is compact, then $\Sigma^2 + (q)$ is a quadratic module with (SMP) \cite{Schmuedgen_book}. We will now provide a statement of the aforementioned generalization of Theorem \ref{thm:10Septm1}.

\begin{cor}
\label{cor:A1}
Let $(q)$ be a non-trivial, real principal ideal of $\RR[x_1, \ldots, x_d]$ such that the quadratic module $\Sigma^2 + (q) \subseteq \RR[x_1, \ldots, x_d]$ has the strong moment property and let $Q \subseteq \RR[x_1, \ldots, x_d]$ be an archimedean quadratic module. Then
the quadratic module $ \Sigma^2 + qQ$ has the strong moment property.
\end{cor}

\begin{proof}
The proof can be carried out in much the same way as the proof of Theorem \ref{thm:10Septm1} with the caveat that one will have to rely on the quadratic module $\Sigma^2 + (q)$ satisfying the (SMP) to obtain a measure supported on $\cV(q)$ for the functional $\Lambda$.
\end{proof}

\bibliographystyle{amsplain}
\bibliography{Bib}

\end{document}